\documentclass[a4paper,11pt]{article}
\usepackage[top=2.5cm, bottom=2.5cm, left=2cm, right=2cm]{geometry}
\usepackage{settings_custom}
\allowdisplaybreaks

\begin{document}

\title{Fitting an ellipsoid to a quadratic number of random points}
\date{\today}
\author{Afonso S.\ Bandeira, Antoine Maillard, Shahar Mendelson, Elliot Paquette}
\maketitle

\begin{abstract}
    We consider the problem $(\rm P)$ of fitting $n$ standard Gaussian random vectors in $\bbR^d$ to the boundary of a centered ellipsoid, as $n, d \to \infty$.
    This problem is conjectured to have a sharp feasibility transition: for any $\eps > 0$, if $n \leq (1 - \eps) d^2 / 4$ then $(\rm P)$ has a solution with high probability, while $(\rm P)$ has no solutions with high probability if
    $n \geq (1 + \eps) d^2 /4$.
    So far, only a trivial bound $n~\geq~d^2 / 2$ is known on the negative side, while the best results on the positive side assume $n \leq d^2 / \plog(d)$.
    In this work, we improve over previous approaches using a key result of Bartl~\&~Mendelson (2022) on the concentration of Gram matrices of random vectors under mild assumptions on their tail behavior.
    This allows us to give a simple proof that $(\rm P)$ is feasible with high probability when $n \leq d^2 / C$, for a (possibly large) constant $C > 0$.
\end{abstract}

\section{Introduction}\label{sec:introduction}
\noindent
We study the following question: given $n$ vectors in $\bbR^d$ independently sampled from the standard Gaussian measure, when does there exist an ellipsoid centered at $0$ 
whose boundary goes through all the vectors? 
This question was raised by \cite{saunderson2011subspace,saunderson2012diagonal,saunderson2013diagonal}, and has received 
significant attention recently \cite{venkat2022near,kane2022nearly,hsieh2023ellipsoid}. We will discuss the motivations behind this problem and review some recent literature in Section~\ref{subsec:motivation}.
In the original series of work of Saunderson\&al \cite{saunderson2011subspace,saunderson2012diagonal,saunderson2013diagonal}, it was conjectured based on numerical experiments that the ellipsoid fitting property undergoes a phase transition in the limit $d \to \infty$ for $n \sim d^2 / 4$. 
Notably, the threshold $d^2/4$ corresponds to the statistical dimension of the cone of positive semidefinite matrices \cite{gordon1988milman,amelunxen2014living} (see \cite{venkat2022near} for a discussion).

\begin{conjecture}[The ellipsoid fitting conjecture]\label{conj:main}
    \noindent
    Let $n, d \geq 1$, and $x_1, \cdots, x_n \iid \mcN(0, \Id_d/d)$.
    Let $p(n, d)$ be defined as the probability of existence of a fitting ellipsoid centered in $0$:
    \begin{equation*}
        p(n, d) \coloneqq \bbP\Big[\exists \Sigma \in \mcS_d \, : \, \Sigma \succeq 0 \quad \textrm{and} \quad x_i^\T \Sigma x_i = 1 \quad (\forall i \in [n])\Big].
    \end{equation*}
    For any $\eps > 0$, the following holds:
    \begin{equation*}
        \begin{dcases}
            \limsup_{d \to \infty} \frac{n}{d^2} \leq \frac{1 - \eps}{4} &\Rightarrow \lim_{d \to \infty} p(n, d) = 1, \\ 
            \liminf_{d \to \infty} \frac{n}{d^2} \geq \frac{1 + \eps}{4} &\Rightarrow \lim_{d \to \infty} p(n, d) = 0. 
        \end{dcases}
    \end{equation*}
\end{conjecture}
\noindent
Our main result gives a positive answer to the existence statement of Conjecture~\ref{conj:main}, 
up to a constant factor in $n/d^2$. We present its proof in Section~\ref{sec:proof}.
\begin{theorem}[Ellipsoid fitting up to a constant]\label{thm:main}
    \noindent
    Let $n, d \geq 1$, and $x_1, \cdots, x_n \iid \mcN(0, \Id_d/d)$.
    Given any $\beta \geq 1$,
    there exist a (small) constant $\alpha = \alpha(\beta) > 0$ and a (large) constant $C = C(\beta) > 0$ such that for $n \leq \alpha d^2$:
    \begin{equation*}
        \bbP[\exists \Sigma \in \mcS_d \, : \, \Sigma \succeq 0 \quad \textrm{and} \quad x_i^\T \Sigma x_i = 1 \quad (\forall i \in [n])] \geq 1 - C n^{-\beta}.
    \end{equation*}
\end{theorem}
\noindent
\textbf{From polynomial to exponential probability bounds --} 
While we show a polynomial lower bound on the probability, as we will notice during the detailing of the proof, 
we believe that such a lower bound can be improved to an exponential lower bound of the type $1 - 2 \exp(-Cd)$, 
for $n \leq \alpha d^2$ and a universal constant $\alpha > 0$. 
We highlight the principles of this improvement in the proof, and detail how it would require a slightly deeper dive into the arguments of the proof of the main result of \cite{bartl2022random}. Since the 
main conjecture of ellipsoid fitting only concerns the limit of the probability and not its scaling, we leave this improvement for future work, 
and will sometimes use probability estimates that are not the sharpest possible, but are sufficient for our goal.

\subsection{Motivation and related literature}\label{subsec:motivation}

We give here a brief overview of the motivations to consider the ellipsoid fitting problem, as well as previous results on this conjecture. 

\myskip
Despite the fact that Conjecture~\ref{conj:main} remains open, the ellipsoid fitting property is a natural question in random geometry.
Notably, if the vectors $x_1, \cdots, x_n$ satisfy this property, then there is no vector~$x_i$ lying in the interior of the convex hull of the other vectors $(\pm x_j)_{j \neq i}$. Moreover, this problem has several connections with machine learning and theoretical computer science, which motivated its introduction. Examples of these connections include the decomposition of a data matrix into a sum of diagonal and low-rank components \cite{saunderson2011subspace,saunderson2012diagonal,saunderson2013diagonal}, overcomplete independent component analysis \cite{podosinnikova2019overcomplete}, 
or the discrepancy of random matrices \cite{saunderson2012diagonal,venkat2022near}. 
Relations to these various problems are discussed more extensively in the introduction of \cite{venkat2022near}, to which we refer the interested reader for more details.

\myskip
\textbf{The negative side of the conjecture --}
A dimension counting argument shows that ellipsoid fitting is generically not possible if $n > d (d+1)/2$, implying that the negative part of Conjecture~\ref{conj:main} is non-trivial only in the range $d^2/4 \lesssim n \lesssim d^2/2$. Despite the simplicity of this argument, $d^2/2$ is still the best-known bound on the negative side of Conjecture~\ref{conj:main}.

\myskip
\textbf{Early results --} In the original works that introduced the ellipsoid fitting conjecture \cite{saunderson2011subspace,saunderson2013diagonal}, it was proven that ellipsoid fitting is feasible with high probability if $n \lesssim \mcO(d^{6/5 - \eps})$ (for any $\eps > 0$). This bound was improved to $n \lesssim \mcO(d^{3/2 - \eps})$ in \cite{ghosh2020sum}, where the result was obtained as a corollary of the proof of a Sum-of-Squares lower bound for the Sherrington-Kirkpatrick Hamiltonian of statistical physics\footnote{In the revised version of \cite{venkat2022near}, as well as in \cite{hsieh2023ellipsoid}, it was noticed that the results of \cite{ghosh2020sum} actually hold for $n \lesssim \mcO(d^2 / \plog(d))$.}, using a pseudo-calibration construction.

\myskip
\textbf{Comparison with recent work --}
Our proof is based on an ``identity perturbation'' construction, an idea which 
was described in \cite{venkat2022near}, and used in \cite{kane2022nearly} to prove that $p(n, d) \to 1$
under the assumption that $n = \mcO(d^2 / \plog(d))$.
On the other hand, \cite{venkat2022near} uses a least-square construction to prove that ellipsoid fitting is possible with high probability 
under the similar condition $n = \mcO(d^2 / \plog(d))$\footnote{We note that \cite{venkat2022near} was recently updated to present an alternative proof through the identity perturbation construction, 
again under the assumption $n = \mcO(d^2 / \plog(d))$.}.

\myskip
Our proof follows in part the one of \cite{kane2022nearly}, improving a crucial operator norm bound thanks to results of \cite{bartl2022random}. As mentioned in \cite{kane2022nearly}, using a suboptimal bound on this operator norm was the main limitation that prevented the authors to prove the existence of a fitting ellipsoid for $n \leq d^2 / C$. 
We emphasize that numerical studies \cite{venkat2022near} suggested that the identity perturbation construction is successful only in the range $n \lesssim d^2/10$, so in order to resolve Conjecture~\ref{conj:main} (or even just the existence part) it appears a new idea is needed\footnote{Numerical simulations of \cite{venkat2022near} suggest the least-squares approach suffers from the same shortcomings.}. 

\myskip 
\textbf{Parallel work --}
As we were finalizing the current manuscript, another proof that ellipsoid fitting is possible at a quadratic number of points was proposed~\cite{hsieh2023ellipsoid}.
Like our approach, the proof in~\cite{hsieh2023ellipsoid} is based on the identity perturbation construction, but the proof techniques appear to us to be quite different:
\cite{hsieh2023ellipsoid} relies on the theory of graph matrices, and as such strengthens similar arguments presented in~\cite{venkat2022near} (while our proof can instead be viewed as a strengthening of the arguments in~\cite{kane2022nearly}).
Finally, shortly after the present work appeared online, a third proof was proposed in the independent work~\cite{tulsiani2023ellipsoid}.

\myskip
More specifically, our approach
relies on obtaining a crucial bound on the operator norm of a kernel Gram matrix by mapping it to the Gram matrix of flattened rank-one matrices,
and using the results of \cite{bartl2022random}. This latter work showed the concentration of the Gram matrix of i.i.d.\ vectors $X_1, \cdots, X_n$ under 
the assumption that the first moments of the projections $\langle X, u\rangle$ satisfy (uniformly in $u$) a $\psi_\alpha$-like tail bound for some $\alpha \in (0, 2]$.

\subsection{The dual semidefinite program}\label{subsec:dual}

Note that ellipsoid fitting belongs to the class of random semidefinite programs, and as such admits a dual formulation. As we find the dual problem to have a particularly interesting formulation we include a short
expository snippet to highlight this dual SDP, and the consequences of Theorem~\ref{thm:main} for it. Namely, it implies the following corollary.
\begin{corollary}[Dual problem]\label{cor:dual}
    \noindent
    Let $n, d \geq 1$, and $x_1, \cdots, x_n \iid \mcN(0, \Id_d/d)$.
    Given any $\beta \geq 1$,
    there exist a (small) constant $\alpha = \alpha(\beta) > 0$ and a (large) constant $C = C(\beta) > 0$ such that for $n\leq \alpha d^2$:
    \begin{equation*}
        \bbP\Bigg[\exists z \in \bbR^n \, : \, \sum_{i=1}^n z_i = 0 \textrm{ and } \lambda_{\max}\Bigg(\sum_{i=1}^n z_i x_i x_i^\T\Bigg) < 0 \Bigg] \leq C n^{-\beta}.
    \end{equation*}
\end{corollary}
\noindent 
Corollary~\ref{cor:dual} rewrites ellipsoid fitting as a problem of ``balancing'' rank-one matrices: we show that for $n \leq \alpha d^2$ it is impossible to find a centered balancing of $(x_i x_i^\T)$ such that the resulting matrix is negative definite (nor positive definite as one can always consider $-z$). We note however that duality doesn't play any explicit role in the proof of Theorem~\ref{thm:main}.

\myskip
\begin{proof}[Proof of Corollary~\ref{cor:dual} --]
Ellipsoid fitting is a semidefinite program, which we can write in the canonical form
\begin{equation*}
    \min_{\substack{\Sigma \succeq 0 \\ \Tr[\Sigma x_i x_i^\T] = 1}} \Tr[A S] \in \{0, +\infty\},
\end{equation*}
with $A = 0$. 
By weak duality and Theorem~\ref{thm:main}, with probability at least $1 - Cn^{-\beta}$ for $n \leq \alpha(\beta) d^2$, 
its dual semidefinite program satisfies:
\begin{equation*}
    \max_{\substack{y \in \bbR^n \\ \sum_{i=1}^n y_i x_i x_i^\T \preceq \, 0}} \sum_{i=1}^n y_i = 0. 
\end{equation*}
We now condition on this event.
Thus for all $y \in \bbR^n$, if $\sum y_i > 0$ then $\lambda_{\max}(\sum_{i=1}^n y_i x_i x_i^\T) > 0$.
Let $z \in \bbR^n$ such that $\sum_{i=1}^n z_i = 0$. To prove Corollary~\ref{cor:dual}, it suffices to show that 
$\lambda_{\max}(\sum_{i=1}^n z_i x_i x_i^\T) \geq 0$.
Let $M(z) \coloneqq \sum_{i=1}^n z_i x_i x_i^\T$.
Let $\eps > 0$, and $y_i(\eps) \coloneqq z_i + \eps$. Since $\sum y_i > 0$, there exists $u_\eps \in \mcS^{d-1}$ (the Euclidean unit sphere in $\bbR^d$) such that 
$u_\eps^\T M(z) u_\eps + \eps \sum_{i=1}^n \langle u_\eps, x_i \rangle^2 > 0$. Extracting a converging sub-sequence as $\eps \to 0$ by compactness, there exists $u \in \mcS^{d-1}$ with $u^\T M(z) u \geq 0$.
\end{proof}

\section{Proof of Theorem~\ref{thm:main}}\label{sec:proof}
\noindent
\textbf{Notation --}
Positive universal constants are generically denoted as $c_k$ or $C_k$, and may vary from line to line.
We will clarify possible dependencies of such constants on relevant parameters when necessary.
$\mcS_d$ denotes the set of $d \times d$ real symmetric matrices, $\Id_d$ is the identity matrix, and
$\ones_d$ is the all-ones vector. $\mcS^{d-1}$ is the Euclidean unit sphere in $\bbR^d$.

\myskip
\textbf{Remark --}
Since the ellipsoid fitting has a clear monotonicity property with respect to $n$, we assume without loss of generality in what follows that $n = \omega(d^{2-\eps})$ for any fixed $\eps > 0$. The polynomial exponent on the probability estimates, of the form $n^{-\beta}$, can be taken to be arbitrarily large, but it will be considered fixed throughout, with $\beta \geq 1$, and as it will be clear below constants generally depend on $\beta$.

\subsection{Identity perturbation ansatz}

In the identity perturbation ansatz \cite{venkat2022near,kane2022nearly}, we look for a fitting ellipsoid $\Sigma \in \mcS_d$ in the form: 
\begin{equation}\label{eq:ip_ansatz}
    \Sigma = \Id_d + \sum_{i=1}^n q_i x_i x_i^\T,
\end{equation}
for some $q \in \bbR^n$.
Having $\Sigma \succeq 0$ is thus equivalent to:
\begin{equation}\label{eq:Sigma_pos}
    \sum_{i=1}^n q_i x_i x_i^\T \succeq - \Id_d.
\end{equation}
We denote $x_i = \sqrt{d_i} \omega_i$, with $\omega_i \iid \Unif[\mcS^{d-1}]$, and $d_i \coloneqq \|x_i\|_2^2$, 
and we let $D \coloneqq \Diag(\{d_i\}_{i=1}^n)$ and $\Theta \in \bbR^{n  \times n}$ with $\Theta_{ij} \coloneqq \langle \omega_i, \omega_j \rangle^2$.
Note that $d_i$ are i.i.d.\ variables, independent of the directions $\omega_i$.
Plugging the ansatz of eq.~\eqref{eq:ip_ansatz} into the ellipsoid fitting equations $x_i^\T \Sigma x_i = 1$ yields:
\begin{equation*}
    \ones_n = D \ones_n + D \Theta D q.
\end{equation*}
Assuming that $D$ and $\Theta$ are invertible, this equation is solved by:
\begin{equation*}
    q = D^{-1} \Theta^{-1} (D^{-1} \ones_n - \ones_n).
\end{equation*}
Plugging it back into eq.~\eqref{eq:Sigma_pos}, we see that the identity perturbation 
ansatz gives a semidefinite positive solution to the ellipsoid fitting problem if $\Theta, D$ are invertible, and
\begin{equation}\label{eq:success_ip}
    \min_{a \in \mcS^{d-1}}\sum_{i=1}^n \Big[\Theta^{-1} (D^{-1} \ones_n - \ones_n)\Big]_i \langle a, \omega_i \rangle^2 \geq -1.
\end{equation}

\subsection{Concentration of a kernel Gram matrix}

We use the following critical lemma on the concentration of the matrix $\Theta$ appearing in eq.~\eqref{eq:success_ip}. 
For $\omega_1, \cdots, \omega_n \iid \Unif[\mcS^{d-1}]$, we call $\Theta_{ij} = \langle \omega_i, \omega_j \rangle^2 = \langle \omega_i \omega_i^\T, \omega_j \omega_j^\T \rangle$ a ``kernel Gram matrix'' since it corresponds to the Gram matrix of $\{\omega_i\}_{i=1}^n$ under the kernel $K(x, y) = \langle x, y \rangle^2$.
A key technical difficulty is that while $\Theta$ can also be seen as the Gram matrix of $\{\omega_i \omega_i^\T\}_{i=1}^n$, 
the random matrices $\omega_i \omega_i^\T$ are not centered, and have tails which are heavier than sub-Gaussian, preventing us from applying classical results on the concentration of Gram matrices of sub-Gaussian random vectors, 
see e.g.~\cite{liaw2017simple}.
\begin{lemma}[Concentration of a kernel Gram matrix]
    \label{lemma:operator_norm_Theta}
    \noindent
    Let $n, d \geq 1$, and $\omega_1, \cdots, \omega_n \iid \mathrm{Unif}[\mcS^{d-1}]$.
    Let $\Theta_{ij} \coloneqq \langle \omega_i, \omega_j \rangle^2$.
    For any $\beta \geq 1$, there are constants such that, with probability greater than 
    $1 - n^{-\beta} - 2 \exp(-c_0 n)$, the following occurs: 
    \begin{equation}
        \label{eq:concentration_Theta}
        \|\Theta - \EE \Theta \|_\op \leq \frac{C_1}{d} + C_2(\beta) \Big(\sqrt{\frac{n}{d^2}} + \frac{n}{d^2}\Big)
    \end{equation}
\end{lemma}
\noindent
Notice that $\EE \Theta = (1-1/d)\Id_n + (1/d) \ones_n \ones_n^\T$.
This lemma is a consequence of the analysis of \cite{bartl2022random}, and is proven in Section~\ref{subsec:proof_lemma_operator_norm}.

\myskip 
\textbf{Remark: improving the probability upper bound --} 
A careful analysis of the proof arguments of \cite{bartl2022random} reveals that in the present case in which the matrix to control is a Gram matrix 
of sub-exponential vectors (which will be the case here as detailed in the proof), the probability estimate could likely be 
improved significantly to yield a probability lower bound of $1 - 2 \exp(-c n)$. We leave for future work to carry out this improvement,
and keep a formulation that follows directly from the results of \cite{bartl2022random}.

\myskip
We get the following corollary:
\begin{corollary}[Concentration of the inverse]\label{cor:concentration_Thetainv}
    \noindent 
    Let $n, d \geq 1$, and $\omega_1, \cdots, \omega_n \iid \mathrm{Unif}[\mcS^{d-1}]$.
    Let $\Theta_{ij} \coloneqq \langle \omega_i, \omega_j \rangle^2$.
    For any $\beta \geq 1$, 
    there exists $\alpha = \alpha(\beta) > 0$ and constants such that 
    if $n \leq \alpha d^2$ and $d \geq d_0(\beta)$, then 
    with probability at least $1 - n^{-\beta} - 2 \exp(-c_0 n)$:
    \begin{equation}\label{eq:concentration_Thetainv}
        \Big\| \Theta^{-1} - \Big(\Id_n - \frac{1}{n} \ones_n \ones_n^\T\Big)\Big\|_\op \leq \frac{C_1}{d} + C_2(\beta) \sqrt{\frac{n}{d^2}} + \frac{d}{n}.
    \end{equation}
    In particular, assuming $n = \omega(d)$, for all $\beta \geq 1$ there is $\alpha = \alpha(\beta) > 0$ such that if $n \leq \alpha d^2$:
    \begin{equation}\label{eq:condition_nb_Theta}
        \bbP[\|\Theta^{-1}\|_\op \leq 2] \geq 1 - 2n^{-\beta}.
    \end{equation}
\end{corollary}
\begin{proof}[Proof of Corollary~\ref{cor:concentration_Thetainv} --]
Note that $\|\EE \Theta - [\Id_n + (1/d) \ones_n \ones_n^\T]\|_\op = (1/d)$, so that eq.~\eqref{eq:concentration_Theta} also holds replacing $\EE \Theta$ by $\Id_n + (1/d) \ones_n \ones_n^\T$.
We use the following elementary lemma, proven in Section~\ref{subsec:proof_lemma_inverse_op}.
\begin{lemma}\label{lemma:inverse_op}
    \noindent  
    Let $A, B \in \mcS_n$ two symmetric matrices, such that $B \succ 0$, and for some $\eps < \lambda_{\min}(B)$ we have 
    $\|A - B \|_\op \leq \eps$. Then 
    \begin{equation*}
       \|A^{-1} - B^{-1}\|_\op \leq \eps \frac{\|B^{-1}\|_\op^2}{1 - \eps \|B^{-1}\|_\op}.
    \end{equation*}
\end{lemma}
\noindent
Applying Lemma~\ref{lemma:inverse_op} to $B = \Id_n + (1/d) \ones_n \ones_n^\T$, such that $\lambda_{\min}(B) = 1$, 
and $B^{-1} = \Id_n - (d+n)^{-1} \ones_n \ones_n^\T $, gives, with probability at least $1 - n^{-\beta} - 2 \exp(-c_0 n)$:
\begin{align*}
    \Big\|\Theta^{-1} - \Big(\Id_n - \frac{1}{n} \ones_n \ones_n^\T\Big)\Big\|_\op &\leq \Big\|\Theta^{-1} - \Big(\Id_n - \frac{1}{n+d} \ones_n \ones_n^\T\Big)\Big\|_\op + \frac{d}{n}, \\
    &\leq \frac{\frac{C_1}{d} + C_2(\beta)\Big(\sqrt{\frac{n}{d^2}} + \frac{n}{d^2}\Big)}{1 - \frac{C_1}{d} - C_2(\beta)\Big(\sqrt{\frac{n}{d^2}} + \frac{n}{d^2}\Big)} + \frac{d}{n}, \\ 
    &\leq \frac{C'_1}{d} + C'_2 \sqrt{\frac{n}{d^2}} + \frac{d}{n},
\end{align*}
for large enough $d$ and small enough $n/d^2$ (depending only on $\beta$).
\end{proof}

\subsection{Reducing to a net}

We show some useful estimates in Section~\ref{subsec:proof_high_p_events}, summarized in the following lemma.
\begin{lemma}[Some high-probability events]\label{lemma:high_p_events}
    \noindent
    Let $\omega_1, \cdots, \omega_n \iid \Unif[\mcS^{d-1}]$, and $\Theta_{ij} \coloneqq \langle \omega_i, \omega_j \rangle^2$.
    Denote $U(a)_i \coloneqq \langle \omega_i, a \rangle^2$ for $a \in \mcS^{d-1}$.
    We let $(a_j)_{j=1}^N$ be a $(1/2)$-net of $\mcS^{d-1}$. 
    Let $\beta \geq 1$.
    There exists $\alpha = \alpha(\beta) > 0$ such that if $n \leq \alpha d^2$,
    then 
    we have:
    \begin{itemize}
        \item[$(i)$] $\bbP[E_1] \geq 1 - 2 \exp(-C_1 d)$, with $E_1 \coloneqq \{\max_{j \in [N]} \|U(a_j)\|_2 \leq C_2\}$ (for a sufficiently large $C_2$).
        \item[$(ii)$] $\bbP[E_2] \geq 1 - 2 n^{-\beta}$, with $E_2 \coloneqq \{\|\Theta^{-1}\|_\op \leq 2\}$.
    \end{itemize}
\end{lemma}
\noindent
In the following, we fix  $(a_j)_{j=1}^N$ a $(1/2)$-net of $\mcS^{d-1}$, such that $N \leq 5^d$ \cite{vershynin2018high}.
Let $\tilde q \coloneqq D^{-1} \ones_n - \ones_n$.
For any matrix $M \in \bbR^{d \times d}$, we have \cite{vershynin2018high}:
\begin{equation*}
    \max_{a \in \mcS^{d-1}} a^\T M a \leq 2 \max_{a \in \mcN} a^\T M a.
\end{equation*}
Therefore:
\begin{equation}\label{eq:necessary_condition_net}
    \bbP[\min_{a \in \mcS^{d-1}}\sum_{i=1}^n (\Theta^{-1} \tilde q)_i \langle a, \omega_i \rangle^2 \leq -1] 
    \leq  
    \bbP\Bigg[\max_{j \in [N]}\Bigg|\sum_{i=1}^n (\Theta^{-1} \tilde q)_i \langle a_j, \omega_i \rangle^2 \Bigg|\geq \frac{1}{2} \Bigg].
\end{equation}
Defining $g_\Theta(a) \coloneqq \sum_{i=1}^n \tilde q_i [\Theta^{-1} U(a)]_i$,
our goal reduced to show that $\max_{j \in [N]} |g_\Theta(a_j)| \leq 1/2$ with probability at least $1 - C n^{-\beta}$, for $n/d^2$ small enough.
First, we show that we can truncate and center the variables $\tilde q_i$:
\begin{lemma}[Truncating and centering $\tilde q$]\label{lemma:truncation_centering}
    \noindent
    Let $A_i \coloneqq \{ |\tilde q_i| \leq 1  \}$ and $A \coloneqq \cap_{i=1}^n A_i$.
    We denote $r_i \coloneqq \tilde q_i | A$, and $y_i \coloneqq r_i - \EE r_i$.
    Then $\{y_i\}_{i=1}^n$ are i.i.d.\ centered $K/\sqrt{d}$-sub-Gaussian random variables, for some universal $K > 0$.
    Moreover, for any $\beta \geq 1$ there exists $\alpha = \alpha(\beta) > 0$ such that if $n \leq \alpha d^2$, then:
    \begin{equation*}
        \bbP\Bigg[\max_{j \in [N]}\Bigg|\sum_{i=1}^n (\Theta^{-1} \tilde q)_i \langle a_j, \omega_i \rangle^2 \Bigg|\geq \frac{1}{2} \Bigg] 
        \leq 
        \bbP\Bigg[\max_{j \in [N]}\Bigg|\sum_{i=1}^n (\Theta^{-1} y)_i \langle a_j, \omega_i \rangle^2 \Bigg|\geq \frac{1}{4} \Bigg] 
        + C n^{-\beta}.
    \end{equation*}
\end{lemma}
\noindent
This lemma is proven in Section~\ref{subsec:proof_truncation_centering}.

\subsection{Controlling points on the net}

In what follows, we replace the variables $\tilde q_i$ by $y_i$ by using Lemma~\ref{lemma:truncation_centering} (assuming $n \leq \alpha d^2$ for $\alpha = \alpha(\beta)$ small enough).
We define, for $a \in \mcS^{d-1}$: 
\begin{equation}\label{eq:f_Theta}
    f_\Theta(a) \coloneqq \sum_{i=1}^n y_i [\Theta^{-1} U(a)]_i = \sum_{i=1}^n [\Theta^{-1}y]_i U(a)_i,
\end{equation}
with $U(a) \coloneqq (\langle \omega_i, a \rangle^2)_{i=1}^n$.
We prove in Section~\ref{subsec:proof_condition_y} the following elementary lemma:
\begin{lemma}\label{lemma:condition_y}
    \noindent
    Let $\{y_i\}_{i=1}^n$ be i.i.d.\ centered sub-Gaussian random variables, with $\|y_1\|_{\psi_2} \leq K / \sqrt{d}$, 
    and $M \in \mcS_n$. 
    Then:
    \begin{equation*}
       \bbP\big[\|M y\|_\infty \geq C \|M\|_\op d^{-3/8}\big] \leq 2n \exp\{-d^{1/4}\}.
    \end{equation*}
\end{lemma}
\noindent
We let
\begin{equation*}
    E_3 \coloneqq \big\{\|\Theta^{-1} y\|_\infty \leq C \|\Theta^{-1}\|_\op d^{-3/8}\big\},
\end{equation*} 
and $E \coloneqq \cap_{k=1}^3 E_k$. 
We have from Lemmas~\ref{lemma:high_p_events} and \ref{lemma:condition_y} that (recall that $y$ is independent of $\Theta$) 
there is $\alpha = \alpha(\beta) > 0$ such that for $n \leq \alpha d^2$: 
\begin{equation}\label{eq:E_high_p_event}
    \bbP[E] \geq 1 - C n^{-\beta}.
\end{equation}
\noindent
Let us fix $a \in \mcS^{d-1}$.
For $\eta \in (0,1)$ we define $S(\eta) \coloneqq \{i \in [n] \, : \, |\langle \omega_i, a \rangle| > \eta \}$. 
Since $\omega_i \iid \Unif[\mcS^{d-1}]$, $|\langle \omega_i , a \rangle|$ are i.i.d.\ sub-Gaussian random variables, with sub-Gaussian norm~$C / \sqrt{d}$~\cite{vershynin2018high}.
$|S(\eta)|$ is thus a binomial random variable, 
with parameters $n$ and $p \leq 2 \exp\{-C d \eta^2\}$.
By Theorem~1 of \cite{klenke2010stochastic}, $|S(\eta)|$ is stochastically dominated by 
a Poisson random variable with parameter $-n \log (1-p)$. 
Assuming that $d \eta^2 \to \infty$, we have for $d$ large enough\footnote{Since $\log(1-x) \geq -2x$ for $0 \leq x \leq 1/2$.}, 
\begin{equation*}
    - n \log (1-p) \leq 2n p \leq 4 n \exp\{-C d \eta^2\}.
\end{equation*}
Letting $\lambda \coloneqq 4 n \exp\{-C d \eta^2\}$ and $X \sim \mathrm{Pois}(\lambda)$, $|S(\eta)|$ is thus stochastically dominated  
by $X$. 
We reach that for all $x > \lambda$ (see e.g.\ Theorem~5.4 of \cite{mitzenmacher2017probability} for the second inequality): 
\begin{equation}\label{eq:ub_S_tail}
    \bbP[|S(\eta)| \geq x] \leq \bbP[X \geq x] \leq \Bigg(\frac{e\lambda}{x}\Bigg)^x e^{-\lambda}.
\end{equation}
We get from eq.~\eqref{eq:ub_S_tail} that 
\begin{equation}\label{eq:ub_S}
    \bbP[|S(\eta)| \geq d^{1/4}] \leq \exp\Big\{d^{1/4} \log (4ne) - C d^{5/4} \eta^2  -\frac{d^{1/4} \log d}{4}\Big\} \leq \exp\Big\{d^{1/4} \log n - C d^{5/4} \eta^2\Big\}.
\end{equation}
We decompose $f_\Theta(a)$ in two parts, which we control separately:
\begin{equation}\label{eq:decomposition_f}
    f_\Theta(a) = \underbrace{\sum_{i \in S(\eta)} [\Theta^{-1}y]_i U(a)_i}_{\eqqcolon f_1(\eta,a)} + \underbrace{\sum_{i \notin S(\eta)} [\Theta^{-1}y]_i U(a)_i}_{\eqqcolon f_2(\eta,a)}.
\end{equation}
First, we have that under the event $E$ of eq.~\eqref{eq:E_high_p_event}, and by the Cauchy-Schwarz inequality:
\begin{equation*}
    |f_1(\eta, a)| \leq C d^{-3/8} \sum_{i \in S(\eta)} \langle \omega_i, a \rangle^2 \leq C d^{-3/8} |S(\eta)|.
\end{equation*}
Let us pick $\eta = d^{-1/8} t$, for some $t \geq 1$ (so that $d \eta^2 \to \infty$).
Using eq.~\eqref{eq:ub_S} in the previous inequality, as well as the law of total probability (and $\bbP[E] \geq 1/2$), we reach:
\begin{equation}
    \label{eq:control_f1}
     \bbP\big[|f_1(d^{-1/8} t, a)| \geq C_1 d^{-1/8}\big| E \big] \leq 
    2\exp\{d^{1/4} \log n - C_2 d t^2\}.
\end{equation}
We now control $f_2(\eta, a)$.
For a random variable $X(\{y_i, \omega_i\})$, we denote $\|X\|_{\psi_2,y}$ the sub-Gaussian norm of the random variable with respect to the randomness of $\{y_i\}$ only (i.e.\ conditioned on the value of $\{\omega_i\}$).
Since $y_i$ are independent of $\{\omega_i\}$ (and thus of the choice of the set $S(\eta)$ and of $\Theta$), 
we get by Hoeffding's inequality (recall that $y_i$ are i.i.d.\ $K/\sqrt{d}$-sub-Gaussian), that for all 
$\{\omega_i\}$:
\begin{equation*}
    \|f_2(\eta, a)\|_{\psi_2,y}^2 \leq \frac{C}{d} \|\Theta^{-1} \widetilde{U}(a)\|_2^2.
\end{equation*}
Here we denoted $\widetilde{U}(a)_i \coloneqq \langle \omega_i, a \rangle^2 \indi\{|\langle \omega_i, a \rangle| \leq \eta\}$.
Therefore:
\begin{equation}
    \label{eq:psi2_f2}
    \|f_2(\eta, a)\|_{\psi_2,y}^2 \leq \frac{C \|\Theta^{-1}\|_\op^2}{d} \sum_{i \notin S(\eta)} \langle \omega_i, a \rangle^4. 
\end{equation}
We can then prove (see Section~\ref{subsec:proof_control_I}):
\begin{lemma}\label{lemma:control_I}
    \noindent
    For all $q \in [1/2, 1]$, there is a constant $C = C(q) > 0$ such that 
    for all $v \geq 0$, and all $\eta \in (0,1)$: 
    \begin{equation*}
        \bbP\Bigg[\sum_{i \notin S(\eta)} \langle \omega_i , a \rangle^4 \geq \frac{n}{d^2}(3+v)\Bigg] \leq 
        2 \exp\Bigg\{-C \min\Bigg(\frac{n d^{2/q} \eta^{4/q}}{d^4 \eta^8} v^2, n^q d^{1-2q} \eta^{2-4q} v^q\Bigg) \Bigg\}.
    \end{equation*}
\end{lemma}

\subsection{Ending the proof}

We detail now how the combination of eq.~\eqref{eq:control_f1} and Lemma~\ref{lemma:control_I} allows to complete the proof.
By Lemma~\ref{lemma:truncation_centering}, our task reduced to show that for a $1/2$-net $(a_j)_{j=1}^N$ of $\mcS^{d-1}$, we have with probability at least $1 - Cn^{-\beta}$, and assuming $n \leq \alpha d^2$ for $\alpha = \alpha(\beta)$ small enough:
\begin{equation}\label{eq:to_show_f_Theta}
    \max_{j \in [N]} |f_\Theta(a_j)| \leq 1/4.
\end{equation}
Recall the decomposition of eq.~\eqref{eq:decomposition_f}. 
We fix $\eta = d^{-1/8} t$, for $t \geq 1$ large enough (not depending on $n, d$) such that eq.~\eqref{eq:control_f1} gives, for $n,d$ large enough:
\begin{equation}
    \label{eq:control_f1_2}
     \bbP\big[|f_1(d^{-1/8} t, a)| \geq C d^{-1/8}\big| E\big] \leq 
     10^{-d}.
\end{equation}
By Lemma~\ref{lemma:control_I} and eq.~\eqref{eq:psi2_f2} we have, choosing $v = 1$ and $q = 3/5$\footnote{This is an arbitrary choice, the only requirement needed is actually that $q \in (1/2,3/4)$.}, 
that for all $x > 0$:
\begin{align*}
     \bbP\Bigg[|f_2(d^{-1/8} t, a)| \geq x \|\Theta^{-1} \|_\op \sqrt{\frac{n}{d^2}} \Bigg] &\leq 
     \EE_\omega \Bigg[\exp\Bigg(- \frac{Cn x^2}{d\sum_{i \notin S(\eta)} \langle \omega_i , a \rangle^4}\Bigg)\Bigg], \\ 
     &\aleq 
     2 \exp\Bigg\{\!\!-C_1 \!\min\!\Bigg(\frac{n}{t^{4/3} \sqrt{d}} , n^{3/5} d^{-3/20} t^{-2/5} \!\Bigg)\! \Bigg\}\!
     + \exp(- C_2d x^2), \\
     &\bleq 
     2 \exp\Big\{-C_1 n^{3/5} d^{-3/20} t^{-2/5} \Big\}
     + \exp(- C_2d x^2), \\
     &\cleq 
     10^{-d}
     + \exp(- C_2d x^2),
\end{align*}
where we used Lemma~\ref{lemma:control_I} in $(\rm a)$ with $v = 1$ and $q = 3/5$ (and bounding $e^{-z} \leq 1$), 
in $(\rm b)$ the fact that $n/\sqrt{d} = \omega(n^{3/5} d^{-3/20})$ since $n = \omega(d)$, 
and finally in $(\rm c)$ we used that $n = \omega(d^{23/12})$, so that we can bound the first term by $10^{-d}$ for $n,d$ large enough.
We fix $x > 0$ large enough (not depending on $n, d$) such that the second term also
satisfies $\exp(-C_2 d x^2) \leq 10^{-d}$.
All in all, we get:
\begin{equation*}
     \bbP\Bigg[|f_2(d^{-1/8} t, a)| \geq C \|\Theta^{-1}\|_\op \sqrt{\frac{n}{d^2}}\Bigg] \leq 2 \times 10^{-d}.
\end{equation*}
And thus:
\begin{equation}
    \label{eq:control_f2}
     \bbP\Bigg[|f_2(d^{-1/8} t, a)| \geq C \sqrt{\frac{n}{d^2}}\Bigg| E\Bigg] \leq 
     \frac{\bbP\Big[|f_2(d^{-1/8} t, a)| \geq C \|\Theta^{-1}\|_\op \sqrt{\frac{n}{d^2}}\Big]}{\bbP[E]} \leq 
     3 \times 10^{-d}
\end{equation}
Notice that the event $E$ of eq.~\eqref{eq:E_high_p_event} is independent of the net. Thus, we have for all $u > 0$: 
\begin{equation}
\label{eq:end_1}
    \bbP\Big[ \max_{j \in [N]} |f_\Theta(a_j)| \geq u\Big] 
    \leq C n^{-\beta} + \bbP\Big[ \max_{j \in [N]} |f_\Theta(a_j)| \geq u \Big| E\Big].
\end{equation}
Combining eqs.~\eqref{eq:control_f1_2} and \eqref{eq:control_f2} with the union bound (recall $N \leq 5^d$) we get:
\begin{equation}
\label{eq:end_2}
    \bbP\Bigg[ \max_{j \in [N]} |f_\Theta(a_j)| \geq C_1 \sqrt{\frac{n}{d^2}} + C_2 d^{-1/8} \Bigg| E \Bigg] 
    \leq 4 \cdot 5^d \cdot 10^{-d} \leq 4 \cdot 2^{-d}.
\end{equation}
By combining eqs.~\eqref{eq:end_1} and eq.~\eqref{eq:end_2},
taking $d$ large enough, and $n/d^2$ small enough, this ends the proof of eq.~\eqref{eq:to_show_f_Theta}, and thus of Theorem~\ref{thm:main}.

\section{Auxiliary proofs}\label{sec:proofs_more}
\subsection{Proof of Lemma~\ref{lemma:operator_norm_Theta}}\label{subsec:proof_lemma_operator_norm}

We use the matrix flattening function, 
for $M \in \mcS_d$:
\begin{align*}
    \flatt(M) &\coloneqq ((\sqrt{2}M_{ab})_{1 \leq a < b \leq d}, (M_{aa})_{a=1}^d) \in \bbR^{d(d+1)/2}, \\ 
    &= ((2 - \delta_{ab})^{1/2} M_{ab})_{a\leq b}.
\end{align*}
It is an isometry: $\langle \flatt(M), \flatt(N) \rangle = \Tr[MN]$.
Note that $\Theta$ is the Gram matrix of the i.i.d.\ vectors $X_i \coloneqq \flatt(x_i x_i^\T) \in \bbR^p$, 
with $p \coloneqq d(d+1)/2$.

\myskip 
\textbf{Centering --}
Note that $\|X_i\|_2 = \|x_i\|_2^2 = 1$.
Moreover, we have\footnote{We identify the matrices and their flattened versions.} $\EE[X_i] = \Id_d / d$, 
and if $Y_i \coloneqq X_i - \EE[X_i]$, then 
$\langle Y_i, Y_j \rangle = \langle X_i, X_j \rangle - 1/d$.
Therefore, we can write 
\begin{equation*}
    \Theta = H + \frac{1}{d} \ones_n \ones_n^\T,
\end{equation*}
with $H_{ij} \coloneqq \langle Y_i , Y_j \rangle$ the Gram matrix of the $(Y_i)_{i=1}^n$. 
We also sometimes denote $H = Y^\T Y$, with $Y$ the matrix whose columns are given by $Y_1, \cdots, Y_n$.
Note that $\EE[\Theta] = (1-1/d)\Id_n + (1/d) \ones_n \ones_n^\T$.
Thus, to prove Lemma~\ref{lemma:operator_norm_Theta} it suffices to show that with the required probability bound:
\begin{equation}\label{eq:to_show_H}
    \| H - \Id_n\|_\op \leq
    \frac{C_1}{d} + C_2(\beta) \Bigg(\sqrt{\frac{n}{d^2}} + \frac{n}{d^2}\Bigg).
\end{equation}
\textbf{Projecting --}
Note that $\langle Y_i, \flatt(\Id_d) \rangle = 0$, so that $Y_i \in \{\flatt(\Id_d)\}^\perp$.
We denote $P$ the orthogonal projector onto $\{\flatt(\Id_d)\}^\perp$, i.e.\  
\begin{equation}\label{eq:def_P}
    P \coloneqq \Id_p - \frac{1}{d} \flatt(\Id_d) \flatt(\Id_d)^\T.
\end{equation}
We remark that $(P Y_i)_{i=1}^n$ are still i.i.d., 
centered, and we have $\langle P Y_i , P Y_j \rangle = \langle Y_i, Y_j \rangle$.

\myskip
\textbf{Rescaling --}
Note that $\EE[Y_i] = 0$, and without loss of generality (up to using the vectors $Y'_i \coloneqq \eps_i Y_i$ 
with $\eps_i \iid \mathrm{Unif}(\{\pm 1\})$, for which the Gram matrix $H'$ satisfies $H' = \mathrm{Diag}(\eps) H \mathrm{Diag}(\eps)$ and has thus the same eigenvalues as $H$) we can assume the $Y_i$ to be symmetric. 

\myskip 
Let us compute the covariance of $Y$.
For $a \leq b$ and $c \leq d$, we have 
\begin{align}
    \nonumber
    \EE[Y_{ab} Y_{cd}] &= [(2 - \delta_{ab})(2-\delta_{cd})]^{1/2} \Bigg[\EE(x_a x_b x_c x_d) - \frac{\delta_{ab} \delta_{cd}}{d^2}\Bigg], \\
    \nonumber
    &\aeq \frac{[(2 - \delta_{ab})(2-\delta_{cd})]^{1/2}}{d^2} \Bigg[\frac{d}{d+2}(\delta_{ab}\delta_{cd} + \delta_{ac} \delta_{bd} + \delta_{abcd}) - \delta_{ab} \delta_{cd}\Bigg], \\ 
    \nonumber
    &= \frac{1}{d^2} \Bigg[\frac{d}{d+2} \Bigg(\delta_{abcd} + [(2 - \delta_{ab})(2-\delta_{cd})]^{1/2} \delta_{ac} \delta_{bd} \Bigg) - \frac{2}{d+2} \delta_{ab} \delta_{cd}\Bigg], \\
    \label{eq:cov_Y}
    &= \frac{2}{d^2} \Bigg[\frac{d}{d+2}\delta_{ac}\delta_{bd} - \frac{1}{d+2} \delta_{ab} \delta_{cd}\Bigg].
\end{align}
In $(a)$ we used the marginals of uniformly sampled random vectors on $\mcS^{d-1}$, which can easily be obtained e.g.\ by using hyperspherical coordinates\footnote{
    The two moments needed are $d^2 \EE[x_1^4] = 3 d/ (2 + d)$ 
    and $d^2 \EE[x_1^2 x_2^2] = d / (d+2)$.
}.
In matrix notation, eq.~\eqref{eq:cov_Y} can be rewritten as:
\begin{align*}
    \EE[Y Y^\T] &= \frac{2}{d^2} \Bigg[\frac{d}{d+2}\Id_p - \frac{1}{d+2} \flatt(\Id_d) \flatt(\Id_d)^\T\Bigg], \\ 
    &= \frac{2}{d(d+2)} P.
\end{align*}
Therefore, if we denote $V_i \coloneqq P Y_i \in \bbR^{p-1}$ the coordinates of $Y_i$ in $\{\flatt(\Id_d)\}^\perp$, we have that 
$\langle V_i, V_j \rangle = \langle Y_i, Y_j \rangle$, and
\begin{equation*}
    \EE[V V^\T] = \frac{2}{d(d+2)} \Id_{p-1}.
\end{equation*}
Denote
\begin{equation}
    \label{eq:def_Sigma}
    \Sigma\coloneqq (p-1) \EE[V V^\T] = \frac{2(p-1)}{d(d+2)} \Id_{p-1} = \Bigg(1 - \frac{1}{d}\Bigg) \Id_{p-1}.
\end{equation}
In particular 
$\|\Sigma - \Id_{p-1}\|_\op \leq (1/d)$.
Letting $Z \coloneqq \Sigma^{-1/2} V$, the vector $Z$ satisfies $\EE[ZZ^\T] = (p-1)^{-1} \Id_{p-1}$, and 
the Gram matrix $H_Z$ of $Z_1, \cdots, Z_n$ satisfies $H - H_Z = Z^\T (\Sigma-\Id_{p-1}) Z$, 
and thus for all $w \in \bbR^{p-1}$:
\begin{align*}
    |w^\T H_Z w - w^\T H w| &= |w^\T Z^\T (\Sigma-\Id_{p-1}) Z w|, \\
    &\leq (1/d) \|Z w \|_2^2, \\ 
    &= (1/d) w^\T H_Z w.
\end{align*}
Therefore, $\| H - H_Z\|_\op \leq (1/d) \|H_Z \|_\op$.
By the triangle inequality, this yields that
\begin{equation}
    \label{eq:relation_H_HZ}
    \| H - \Id_n \|_\op \leq \frac{1}{d} + \Big(1 + \frac{1}{d}\Big) \|H_Z - \Id_n\|_\op.
\end{equation}
Using eq.~\eqref{eq:to_show_H} and eq.~\eqref{eq:relation_H_HZ}, it is clear that 
we conclude to eq.~\eqref{eq:concentration_Theta}, it is enough to show that (with the required probability bound):
\begin{equation}\label{eq:to_show_HZ}
    \| H_Z - \Id_n \|_\op \leq \frac{6}{d} + C(\beta) \Bigg(\sqrt{\frac{n}{d^2}} + \frac{n}{d^2}\Bigg).
\end{equation}

\myskip
\textbf{Gram matrix estimation --}
We will use the results of \cite{bartl2022random}. 
We need to introduce the definition of a well-behaved random vector:
\begin{definition}[Well-behaved vector]\label{def:well_behaved}
    \noindent
    Let $q \geq 1$.
    A random vector $X \in \bbR^q$ is said to be well-behaved for $n \geq 1$ with constants $L, R > 0$, $\alpha \in (0,2]$, $\delta \in [0,1]$ and $\gamma \in [0,1)$ if:
    \begin{itemize}
        \item[$(i)$] $X$ is symmetric and isotropic: $\EE[XX^\T] = \Id_q$.
        \item[$(ii)$] If one considers $n$ i.i.d.\ draws $X_1, \cdots, X_n$, then with probability at least $1 - \gamma$: 
        \begin{equation*}
            \max_{1 \leq i \leq n} \Bigg|\frac{\|X_i\|_2^2}{q} - 1\Bigg| \leq \delta.
        \end{equation*}
        \item[$(iii)$] For all $2 \leq k \leq R \log n$ and all $t \in \bbR^q$:
        \begin{equation*}
            \|\langle X, t \rangle\|_{L_k} \leq L k^{1/\alpha}  \|\langle X, t \rangle\|_{L_2} = L k^{1/\alpha}  \|t\|_2.
        \end{equation*}
    \end{itemize}
\end{definition}
\noindent
Condition $(iii)$ corresponds to some $\psi_\alpha$ behavior of the projections, uniformly in $t$, and for some $\alpha \in (0,2]$, 
but only up to moments $k = \mcO(\log n)$.
We can now state an immediate corollary to Theorem~1.5 of \cite{bartl2022random} (precisely the particular case corresponding to $T$ being the unit sphere):
\begin{corollary}[\cite{bartl2022random}]\label{thm:shahar}
    \noindent
    Let $n, q \geq 1$.
    Let $\beta \geq 1$.
    Assume that the random vector $A\in \bbR^q$ is well-behaved with respect to $n$ according to Definition~\ref{def:well_behaved}, 
    with constants $L, R = R(\beta), \alpha, \gamma, \delta$. 
    Let $M \in \bbR^{q \times n}$ be a matrix with i.i.d.\ columns $A_1, \cdots A_n$.
    Then, with probability at least $1 - \gamma - 2\exp(- c_0 n) - n^{-\beta}$, we have 
    \begin{equation*}
        \Bigg\|\frac{1}{q}M^\T M - \Id_n\Bigg\|_\op \leq 2 \delta + c(L, \alpha, \beta) \Bigg(\sqrt{\frac{n}{q}} + \frac{n}{q}\Bigg).
    \end{equation*}
\end{corollary}
\noindent
Corollary~\ref{thm:shahar} is an application of Theorem~1.5 of \cite{bartl2022random}, 
for the simplest case in which $T = \mcS^{n-1}$, so that the Gaussian width is $\ell_\star(T) \coloneqq \EE\|g\|_2 \simeq \sqrt{n}$ (for $g \sim \mcN(0, \Id_n)$), 
$d_T \coloneqq \sup_{t \in \mcS^{n-1}}\|t\| = 1$, and $k_\star(T) \coloneqq (\ell_\star(T)/d_T)^2 \simeq n$.
More precisely, we have $(1+ \mcO(n^{-1})) n \leq n^2 / (n+1) \leq k_\star(T) \leq n$.
Note as well that we added the factor $p^{-1}$ in front of the Gram matrix $M^\T M$ (it is implicit in \cite{bartl2022random} because the columns of $M$ there are $A_i / \sqrt{p}$).

\myskip
\textbf{An important remark --}
We emphasize a technical point, related to the final probability bounds we obtain in Theorem~\ref{thm:main}.
In what follows, we will apply Corollary~\ref{thm:shahar} with $R = \infty$, as the moment bound will be valid for all orders. In this context, the analysis of \cite{bartl2022random} would naturally imply that Corollary~\ref{thm:shahar} holds with probability at least $1 - \gamma - 2 \exp(-c_0 n)$, and with a constant $c(L, \alpha)$ not depending on $\beta$. In turn, a more careful analysis would reveal that the probability bound of Theorem~\ref{thm:main} can be made exponentially small in $d$. However, as proving this would require a possibly lengthy technical analysis of the arguments of \cite{bartl2022random}, for reasons of clarity we chose to restrict to the most direct application of Theorem~1.5 of \cite{bartl2022random}, which gives then a suboptimal polynomial probability upper bound.

\myskip
In order to deduce eq.~\eqref{eq:to_show_HZ} from Corollary~\ref{thm:shahar}, with the dimension $q = p-1$ (recall $p = d(d+1)/2$), 
we need to verify that the distribution of the columns $Z_i$ is well-behaved for some $\alpha, L, R, \delta, \gamma$.
We let $A_i \coloneqq \sqrt{q} Z_i$, and we check that it satisfies Definition~\ref{def:well_behaved}.

\myskip 
\textbf{Condition~($i$) -- } Because of the random sign that we can add wlog, we have seen that the distribution of $A$ is symmetric.
Moreover, by our analysis above, $\EE[A A^\T] = q \EE[ZZ^\T] = \Id_q$, so that $A$ is isotropic.

\myskip 
\textbf{Condition~($ii$) -- } 
Notice that for all $i$, $\|Y_i\|_2^2 = \|V_i\|_2^2 = 1 - 1/d$.
Thus, with the notations from above:
\begin{align*}
    \Bigg|\frac{1}{q}\|A\|_2^2 - 1\Bigg| &= \Bigg|V^\T (\Sigma^{-1} - \Id_q) V + \frac{1}{d}\Bigg|, \\ 
    &\leq \|V\|_2^2 \| \Sigma^{-1} - \Id_q \|_\op + \frac{1}{d}, \\ 
    &\aleq \frac{3}{d}.
\end{align*}
In $(\rm a)$ we used that 
\begin{equation*}
    \|\Sigma - \Id_q\|_\op \leq \frac{1}{d} \Rightarrow \|\Sigma^{-1} - \Id_q\|_\op \leq \frac{\frac{1}{d}}{1 - \frac{1}{d}} \leq \frac{2}{d}.
\end{equation*}
Thus, $A$ satisfies the condition $(ii)$ with $\gamma = 0$ and $\delta = 3/d$ (since the bound is deterministic, there is no need to consider $n$ i.i.d.\ samples).

\myskip 
\textbf{Condition~($iii$) -- } 
We are going to see that it actually holds for all $k \geq 2$ with $\alpha = 1$, i.e.\ the random vector $A$ is uniformly sub-exponential.
Let $t \in \bbR^q$.
Then\footnote{Again, since $\|\Sigma-\Id_q\|_\op \leq 1/d \Rightarrow \|\Sigma^{-1/2}-\Id_q\|_\op \leq 2/d$.}: 
\begin{align*}
    |\langle A, t \rangle - \sqrt{q} \langle V, t\rangle| &= |\sqrt{q} V^\T (\Sigma^{-1/2} - \Id_q) t| , \\ 
    &\leq \sqrt{q} \|V\|_2  \times \frac{2}{d} \times \|t\|_2, \\ 
    &\aleq C \|t\|_2,
\end{align*}
using in $(\rm a)$ that $q + 1 = d(d+1)/2$ and that $\|V\|_2^2 = \|Y\|_2^2 = \Tr[(xx^\T - \Id_d/d)^2] = 1 - 1/d \leq 1$.
We have then for all $k \geq 2$: 
\begin{align*}
    \|\langle A, t \rangle\|_k &\aleq 2\Bigg[q^{k/2}\|\langle V, t \rangle\|_k^k + C^k \|t\|_2^k\Bigg]^{1/k} , \\
     &\bleq 2\Bigg[\sqrt{q}\|\langle V, t \rangle\|_k + C \|t\|_2\Bigg],
\end{align*}
using in $(\rm a)$ that $(x+y)^k \leq 2^{k-1}(x^k + y^k)$ for $x, y > 0$, and in $(b)$ Minkowski's inequality 
$(x+y)^{1/k} \leq (x^{1/k} + y^{1/k})$.
Therefore, it is enough to check that for all $k \geq 2$:
\begin{equation}\label{eq:to_check_condition_iii}
    \|\langle V, t \rangle\|_k \leq \frac{L}{d} k^{1/\alpha} \|t\|_2,
\end{equation}
for some $\alpha \in (0,2]$.
We will use the Hanson-Wright inequality for random vectors on the sphere: 
\begin{lemma}[Hanson-Wright]\label{lemma:Hanson_Wright}
    \noindent
    Let $d \geq 1$ and $x \sim \Unif(\mcS^{d-1})$.
    For any $M \in \mcS_d$ and any $u > 0$: 
    \begin{equation*}
        \bbP\Big[\Big|d x^\T M x - \Tr[M]\Big| \geq u\Big] \leq 2 \exp\Bigg\{-C \min \Bigg(\frac{u^2}{\|M\|_F^2}, \frac{u}{\|M\|_\op}\Bigg)\Bigg\}.
    \end{equation*}
\end{lemma}
\noindent
\textbf{Remark --}
We prove Lemma~\ref{lemma:Hanson_Wright} as a consequence of a general Hanson-Wright inequality for random 
vectors satisfying a convex Lipschitz concentration property \cite{adamczak2014note}, easily satisfied by the Haar measure on $\mcS^{d-1}$.
We give details in Section~\ref{subsec:hanson_wright_sphere}.

\myskip 
Recall that $V = P Y \in \bbR^q$, with $P$ the orthogonal projector onto $\flatt(\Id_d)^\perp$, 
and that $t \in \bbR^q$.
If we identify $t$ with the corresponding element of $\bbR^p$ (or the corresponding $d \times d$ symmetric matrix),
then $\Tr[t] = 0$, and
$\langle V, t \rangle = \langle Y, t \rangle = x^\T t x - \Tr[t]/d = x^\T t x$ for $x \sim \Unif(\mcS^{d-1})$.
Using Lemma~\ref{lemma:Hanson_Wright} with $M = t$ gives: 
\begin{equation*}
    \bbP\Big[d |\langle V , t \rangle| \geq u\Big] \leq 2 \exp\Bigg\{-C \min\Bigg(\frac{u^2}{\|t\|_2^2}, \frac{u}{\|t\|_\op}\Bigg)\Bigg\}.
\end{equation*}
It is now classical to deduce the moments from the tails: 
\begin{align*}
    d^k \|\langle V, t \rangle\|_k^k &= \int_0^\infty k u^{k-1} \bbP[d|\langle V, t \rangle| \geq u] \rd u, \\ 
    &\leq 2 k \int_0^\infty u^{k-1} \exp\Bigg\{-C \min\Bigg(\frac{u^2}{\|t\|_2^2}, \frac{u}{\|t\|_\op}\Bigg)\Bigg\}\rd u, \\
    &\leq 2 k \int_0^{\|t\|_2^2/\|t\|_\op} u^{k-1} \exp\{-C u^2 / \|t\|_2^2\}\rd u + 2 k \int_{\|t\|_2^2/\|t\|_\op}^\infty u^{k-1} \exp\{-C u / \|t\|_\op\}\rd u, \\
    &\leq 2 k \int_0^{\infty} u^{k-1} \exp\{-C u^2 / \|t\|_2^2\}\rd u + 2 k \int_{0}^\infty u^{k-1} \exp\{-C u / \|t\|_\op\}\rd u, \\ 
    &\leq k C^{-k/2} \|t\|_2^k \Gamma\Big[\frac{k}{2}\Big] + 2k \Big(\frac{\|t\|_\op}{C}\Big)^k \Gamma(k), \\
    &\leq k \|t\|_2^k \Bigg\{ C^{-k/2} \Gamma\Big[\frac{k}{2}\Big] + 2 C^{-k}\Gamma(k)\Bigg\},
\end{align*}
since $\|t\|_\op \leq \|t\|_2 = \|t\|_F$. This is simply the sum of the sub-Gaussian and sub-exponential part of the tail given by Hanson-Wright's inequality.
Thus we have 
\begin{equation*}
    d \|\langle V, t \rangle\|_k \leq L k \|t\|_2, 
\end{equation*}
which is exactly eq.~\eqref{eq:to_check_condition_iii} for $\alpha = 1$.

\myskip 
Applying Corollary~\ref{thm:shahar} to $A = \sqrt{q}Z$ with $L, R = \infty, \alpha = 1, \gamma = 0, \delta = 3/d$, we reach that for 
all $\beta \geq 1$: 
\begin{equation*}
    \|Z^\T Z - \Id_n\|_\op \leq \frac{6}{d} + C_1(\beta) \Bigg(\sqrt{\frac{n}{d^2}} + \frac{n}{d^2}\Bigg),
\end{equation*}
with probability at least $1 - n^{-\beta} - 2 \exp(-c_0 n)$.
This implies eq.~\eqref{eq:to_show_HZ} and concludes the proof. \qedsymbol

\subsection{Proof of Lemma~\ref{lemma:Hanson_Wright}}\label{subsec:hanson_wright_sphere}

We use a generalization of Hanson-Wright's inequality (usually stated for i.i.d.\ sub-Gaussian vectors) which is due to \cite{adamczak2014note}.
\begin{definition}[Convex concentration property]\label{def:cvx_concentration}
    \noindent
    Let $n \geq 1$ and $X$ be a random vector in $\bbR^n$. We say that $X$ has the convex concentration property with constant $K$ if, for all 
    $\varphi : \bbR^n \to \bbR$ convex and $1$-Lipschitz, we have $\EE|\varphi(X)| < \infty$, and for every 
    $t > 0$: 
    \begin{equation*}
        \bbP[|\varphi(X) - \EE[\varphi(X)]| \geq t] \leq 2 \exp(-t^2/K^2).
    \end{equation*}
\end{definition}
\noindent
Note that if $X = \sqrt{d} x$, with $x \sim \Unif[\mcS^{d-1}]$, 
then $X$ satisfies Definition~\ref{def:cvx_concentration} for some absolute constant $K > 0$ (the function $\varphi$ does not even need to be convex), 
it is one of the most classical results of concentration of measure, cf.\ e.g.\ Theorem~5.1.4 of \cite{vershynin2018high}.
The main result of \cite{adamczak2014note} is the following:
\begin{proposition}[Hanson-Wright \cite{adamczak2014note}]\label{prop:hanson_wright}
    \noindent
    Let $n \geq 1$ and $X$ be a zero-mean vector in $\bbR^n$ that has the convex concentration property with constant $K$.
    Then for all symmetric $M \in \bbR^{n \times n}$ and $t > 0$: 
    \begin{equation*}
        \bbP[|X^\T M X - \EE(X^\T M X)| \geq t] \leq 2 \exp\Bigg(-C \min \Bigg(\frac{t^2}{2K^4 \|M\|_F^2}, \frac{t}{K^2 \|M\|_\op}\Bigg)\Bigg).
    \end{equation*}
\end{proposition}
\noindent
Applying Proposition~\ref{prop:hanson_wright} to the vector $X$ described above yields Lemma~\ref{lemma:Hanson_Wright}.

\subsection{Tail bounds for \texorpdfstring{$\chi^2$}{chi2} random variables}

The following is a useful tail bound on $\chi_d^2$ random variables, from \cite{laurent2000adaptive}. 
\begin{lemma}[Tail bounds for $\chi_d^2$]\label{lemma:tail_bound_chi_square}
    \noindent 
    Let $d \geq 1$, and $x_1, \cdots, x_d \iid \mcN(0,1)$. Let $z \coloneqq (1/d) \sum_{i=1}^d x_i^2$.
    Then for all $u \geq 0$: 
    \begin{equation*}
        \begin{dcases}
            \bbP\Bigg[z - 1 \geq 2 \sqrt{\frac{u}{d}} + 2 \frac{u}{d}\Bigg] &\leq \exp(-u), \\
            \bbP\Bigg[z - 1 \leq -2 \sqrt{\frac{u}{d}} \Bigg] &\leq \exp(-u).
        \end{dcases}
    \end{equation*} 
\end{lemma}

\begin{corollary}[Tail bounds for $\tilde q$]\label{cor:tail_bound_qtilde}
    \noindent
    Let $d \geq 1$, and $x_1, \cdots, x_d \iid \mcN(0,1)$. Let $z \coloneqq (1/d) \sum_{i=1}^d x_i^2$, 
    and we denote $\tilde q \coloneqq 1/z - 1$.
    Notice that $\tilde q \geq -1$.
    Then, for all $t \in (0,1)$: 
    \begin{equation*}
            \bbP[|\tilde{q}| \geq t] \leq 2 \exp\Bigg(-\frac{d t^2}{16}\Bigg).
    \end{equation*} 
\end{corollary}
\begin{proof}[Proof of Corollary~\ref{cor:tail_bound_qtilde} --]
    We start with the upper tail $\tilde q \geq t$.
    Notice that $\tilde q \geq t \Leftrightarrow z \leq (1+t)^{-1}$.
    Using Lemma~\ref{lemma:tail_bound_chi_square} with $4 u = d [t / (1+t)]^2$, we have (using that $t < 1$):
    \begin{equation*}
        \bbP[\tilde q \geq t] \leq \exp\Bigg\{-\frac{d t^2}{4(1+t)^2} \Bigg\} \leq \exp\Bigg\{-\frac{d t^2}{16} \Bigg\}.
    \end{equation*}
    Similarly, for the lower tail, $\tilde q \leq -t \Leftrightarrow z \geq (1-t)^{-1}$.
    Using Lemma~\ref{lemma:tail_bound_chi_square} with $2 u = d [1 / (1-t) - \sqrt{(1+t)/(1-t)} ]$, we have (again using that $t \in (0,1)$):
    \begin{equation*}
        \bbP[\tilde q \leq -t] \leq \exp\Bigg\{-\frac{d}{2} \Bigg[\frac{1}{1-t} - \sqrt{\frac{1+t}{1-t}}\Bigg] \Bigg\} \leq \exp\Bigg\{-\frac{d t^2}{4}\Bigg\}.
    \end{equation*}
    This ends the proof.
\end{proof}

\subsection{Proof of Lemma~\ref{lemma:inverse_op}}\label{subsec:proof_lemma_inverse_op}

\begin{proof}
Note that $\lambda_{\mathrm{min}}(A) \geq \lambda_\mathrm{min}(B) - \eps$, so that $A \succ 0$ and $\|A^{-1}\|_\op \leq \|B^{-1}\|_\op / (1-\eps \|B^{-1}\|_\op)$.
We can use the standard estimate: 
\begin{equation*}
    \|A^{-1} - B^{-1}\|_\op =  \|B^{-1}(B - A)A^{-1}\|_\op \leq \|B^{-1}\|_\op \|A-B\|_\op \|A^{-1}\|_\op.
\end{equation*}
Using the remark above and the fact that $\|A - B\|_\op \leq \eps$ completes the proof.
\end{proof}

\subsection{Proof of Lemma~\ref{lemma:high_p_events}}\label{subsec:proof_high_p_events}

The probability bound for the event $E_2$ is the conclusion of Corollary~\ref{cor:concentration_Thetainv}, 
so we focus on the bound for $E_1$.
To control $\|U(a)\|_2$, we make use of the following tail bound \cite{talagrand1994supremum,hitczenko1997moment,adamczak2011restricted}.
\begin{lemma}[Tail of sum of i.i.d.\ sub-Weibull random variables \cite{adamczak2011restricted}]
    \label{lemma:tail_sum_sub_Weibull}
    \noindent 
    Let $q \in [1/2,1]$, and $W_1, \cdots, W_n$ be i.i.d.\ centered random variables satisfying $\bbP[|W_1| \geq t] \leq C_1 e^{-C_2 t^q}$.
    Then for all $t > 0$:
    \begin{equation*}
        \bbP \Big[\Big|\frac{1}{n}\sum_{\mu=1}^n W_\mu\Big| \geq t\Big] \leq 2 \exp\Big\{-C(q) \min(n t^2, (nt)^q)\Big\}.
    \end{equation*}
\end{lemma}
\noindent
Lemma~\ref{lemma:tail_sum_sub_Weibull} is a generalization of Bernstein's inequality for $\psi_q$ tails, with $q \in [1/2,1]$.
This lemma is stated in \cite{adamczak2011restricted}, see Lemma~3.7 and eq.~(3.7) there, and is a classical consequence 
of the same result for symmetric Weibull random variables \cite{hitczenko1997moment}.

\myskip 
We fix $a \in \mcS^{d-1}$. Note that:
\begin{equation*}
    \|U(a)\|_2^2 = \sum_{i=1}^n \langle \omega_i, a \rangle^4.
\end{equation*}
Since $\langle \omega_i, a \rangle \deq (\omega_i)_1$ by rotation invariance of the Haar measure on $\mcS^{d-1}$, it is easy to check that 
$\EE[\langle \omega_i, a \rangle^4] = (3/d^2)\cdot d/(2+d) \leq 3/d^2$.
Moreover, we have for all $t \geq 0$ \cite{vershynin2018high}:
\begin{equation*}
    \bbP[\langle \omega_i, a \rangle^4 \geq t]\leq 2 \exp\{-C d \sqrt{t}\}.
\end{equation*}
Therefore, applying Lemma~\ref{lemma:tail_sum_sub_Weibull} and using the union bound (recall $N \leq 5^d$), we get: 
\begin{equation*}
    \bbP\Bigg[\sup_{j \in [N]}\|U(a_j)\|_2^2 \geq \frac{3n}{d^2} + t\Bigg] \leq 2 \exp\Bigg\{d \log 5 -C \min\Bigg(\frac{d^4 t^2}{n}, d \sqrt{t}\Bigg)\Bigg\}.
\end{equation*}
Taking e.g.\ $t = (2 \log 5 / C)^2$, and since $d^4/n = \omega(d)$, we reach the conclusion.

\subsection{Proof of Lemma~\ref{lemma:truncation_centering}}\label{subsec:proof_truncation_centering}

Note that $\tilde q_i = 1 / d_i - 1 \deq d/\chi_d^2 - 1$. 
We let $r_i \coloneqq \tilde q_i | A_i$.
The $A_i$ are independent, and by Corollary~\ref{cor:tail_bound_qtilde}, $\bbP[A_i] \geq 1 - 2 \exp(-d/16)$. 
By the law of total expectation and the union bound, we thus have: 
\begin{equation}\label{eq:truncation_1}
    \bbP\Bigg[\max_{j \in [N]}\Bigg|\sum_{i=1}^n (\Theta^{-1} \tilde q)_i \langle a_j, \omega_i \rangle^2 \Bigg|\geq \frac{1}{2} \Bigg] 
    \leq 
    \bbP\Bigg[\max_{j \in [N]}\Bigg|\sum_{i=1}^n (\Theta^{-1} r)_i \langle a_j, \omega_i \rangle^2 \Bigg|\geq \frac{1}{2} \Bigg] 
    + 2n e^{-d/16}.
\end{equation}
Since $\tilde q_i \geq -1$, for all $x \in \bbR$: $\bbP[r_i \leq x] = \bbP[\tilde q_i \leq x \wedge 1] / \bbP[\tilde q_i \leq 1]$, 
and thus for all $x \in (0,1)$, by Corollary~\ref{cor:tail_bound_qtilde}:
\begin{equation*}
    \begin{dcases}
        \bbP[r_i \geq x] &\leq \bbP[\tilde q_i \geq x] \leq 2 e^{-dx^2/16}, \\ 
        \bbP[r_i \leq -x] &\leq \frac{\bbP[\tilde q_i \leq -x]}{1 - 2 e^{-d/16}} \leq 4 e^{- dx^2/16}.
    \end{dcases}
\end{equation*}
Moreover, $\bbP[|r_i| > 1] = 0$.
$r_i$ are thus i.i.d.\ sub-Gaussian random variables, with sub-Gaussian norm smaller than $K / \sqrt{d}$.
Moreover, 
by the law of total expectation: 
\begin{equation*}
    \EE[\tilde q_i] = \EE[r_i] \bbP(A_i) + \EE[\tilde q_i \indi\{|\tilde q_i| \geq 1\}],
\end{equation*}
so that since $\bbP[A_i] \geq 1 - 2 e^{-d/16}$, and using Cauchy-Schwarz: 
\begin{align*}
    |\EE[\tilde q_i] - \EE[r_i]| &\leq |\EE[r_i]| \cdot 2e^{-d/16} + \sqrt{2}\EE[\tilde q_i^2]^{1/2} e^{-d/32}, \\
     &\aleq 2 e^{-d/16} + C e^{-d/32} /\sqrt{d},
\end{align*}
using in $(\rm a)$ that $|r_i| \leq 1$ and that $\EE[\tilde q_i^2]^{1/2} \leq C/\sqrt{d}$.
Since $\EE \tilde q_i = 2 / (d-2)$, we get 
\begin{equation*}
    |\EE r_i| \leq \frac{3}{d}.
\end{equation*}
Recall that $y_i = r_i - \EE r_i$.
Therefore we have, for all $a \in \mcS^{d-1}$: 
\begin{align*}
   \Bigg|\sum_{i=1}^n [\Theta^{-1}(y - r)]_i \langle \omega_i, a \rangle^2\Bigg| &\leq 
   \|\EE r\|_2\|\Theta^{-1}\|_\op \|U(a)\|_2, \\
    &\leq \frac{3 \sqrt{n} }{d} \|\Theta^{-1}\|_\op \| U(a)\|_2.
\end{align*}
Using Lemma~\ref{lemma:high_p_events}, it is clear that if $n \leq \alpha d^2$ for $\alpha = \alpha(\beta) > 0$ small enough, we have 
\begin{equation}\label{eq:truncation_2}
    \bbP\Bigg[\max_{j \in [N]}\Bigg|\sum_{i=1}^n (\Theta^{-1} r)_i \langle a_j, \omega_i \rangle^2 \Bigg|\geq \frac{1}{2} \Bigg] 
    \leq 
    \bbP\Bigg[\max_{j \in [N]}\Bigg|\sum_{i=1}^n (\Theta^{-1} y)_i \langle a_j, \omega_i \rangle^2 \Bigg|\geq \frac{1}{4} \Bigg] 
    + C n^{-\beta}.
\end{equation}
Combining eqs.~\eqref{eq:truncation_1} and eq.~\eqref{eq:truncation_2} gives the sought result.
Finally, $(y_i)_{i=1}^n$ are i.i.d.\ centered sub-Gaussian random variables with sub-Gaussian norm $K / \sqrt{d}$. $\qed$

\subsection{Proof of Lemma~\ref{lemma:condition_y}}\label{subsec:proof_condition_y}

Let $M \in \mcS_n$, and denote $z \coloneqq M y$.
By Hoeffding's inequality, for all $i \in [n]$: 
\begin{equation*}
    \bbP[|z_i| \geq t] \leq 2 \exp\Bigg\{-\frac{Cdt^2}{\|M_i\|_2^2}\Bigg\}
    \leq 2 \exp\Bigg\{-\frac{Cdt^2}{\|M\|_\op^2}\Bigg\},
\end{equation*}
with $(M_i)_{i=1}^n$ the rows of $M$,
since $\|M\|_\op \geq \max_{i \in [n]} \|M_i\|_2$.
Thus by the union bound: 
\begin{equation*}
    \bbP[\|z\|_\infty \geq t] \leq 
    2 n \exp\Bigg\{-\frac{Cdt^2}{\|M\|_\op^2}\Bigg\}.
\end{equation*}
Letting $t = C \|M\|_\op d^{-3/8}$ ends the proof. $\qed$

\subsection{Proof of Lemma~\ref{lemma:control_I}}\label{subsec:proof_control_I}

 Let $q \in [1/2,1]$.
Recall that
\begin{equation*}
    \sum_{i \notin S(\eta)} \langle \omega_i , a \rangle^4
    = \sum_{i = 1}^n \langle \omega_i , a \rangle^4 \indi\{| \langle \omega_i, a \rangle| \leq \eta\}
\end{equation*}
We let $z_i \coloneqq \langle \omega_i , a \rangle^4 \indi\{| \langle \omega_i, a \rangle| \leq \eta\}$. 
They are i.i.d.\ random variables, with $\EE[z_i] \leq \EE[\langle \omega_i, a \rangle^4] \leq 3/d^2$, 
and for all $t \geq 0$:
\begin{align*}
    \bbP[z_i \geq t] &\leq \min \Big[2 e^{-Cd \sqrt{t}}, \indi\{t^{1/4} \leq \eta\}\Big], \\
    &\leq 2 \exp\Big\{-C d \eta^{2 - 4 q} t^q\Big\}.
\end{align*}
Consequently $z'_i = z_i d^{1/q} \eta^{2/q - 4}$ satisfy $\bbP[z'_i \geq t]\leq 2\exp\{-C t^{q}\}$.
We use again Lemma~\ref{lemma:tail_sum_sub_Weibull} to get:
\begin{equation*}
    \bbP\Bigg[\sum_{i=1}^n z_i \geq n \EE[z_i] +  n  d^{-1/q} \eta^{-2/q + 4} t \Bigg] \leq 
    2 \exp\{- C_q \min(nt^2, (nt)^q)\}.
\end{equation*}
This last inequality can be rewritten as, for all $v \geq 0$:
\begin{equation*}
    \bbP\Bigg[\sum_{i=1}^n z_i \geq \frac{n}{d^2} (3+v) \Bigg] \leq 
    2 \exp\Big\{- C_q \min\Big(n d^{-4+2/q} \eta^{4/q-8} v^2, n^q d^{1-2q} \eta^{2-4q} v^q\Big)\Big\}.
    \qed
\end{equation*}

\myskip 
\textbf{Acknowledgements --}
The authors are grateful to Tim Kunisky, from whom they learned about this problem, and to Joel Tropp for insightful discussions.
We would also like to thank anonymous referees for useful suggestions.

\bibliographystyle{alpha}
\bibliography{refs}

\end{document}